\documentclass[11pt]{article}

\let\oldmarginpar\marginpar
\renewcommand\marginpar[1]{\-\oldmarginpar[\raggedleft\footnotesize #1]
{\raggedright\footnotesize #1}}
\usepackage{mathptmx}
\usepackage{amsmath}
\usepackage{amscd}
\usepackage{amssymb}
\usepackage{amsthm}
\usepackage{xspace}
\usepackage[all,tips]{xy}
\usepackage[dvips]{graphicx}
\usepackage{verbatim}
\usepackage{syntonly}
\usepackage{hyperref,color}

\theoremstyle{plain}
\newtheorem{thm}{Theorem}[section]
\newtheorem{cor}[thm]{Corollary}
\newtheorem{prop}[thm]{Proposition}
\newtheorem{lemma}[thm]{Lemma}

\newtheorem{conj}{Conjecture}

\theoremstyle{definition}
\newtheorem{rem}{Remark}

\DeclareMathOperator{\SL}{SL} 
\DeclareMathOperator{\GL}{GL}

\DeclareMathOperator{\D}{D}

\DeclareMathOperator{\Tr}{Tr}
\DeclareMathOperator{\Hom}{Hom}

\DeclareMathOperator{\WDep}{F}
\DeclareMathOperator{\Conj}{Conj}
\DeclareMathOperator{\Con}{CD}

\newcommand{\vp}{\varphi}

\newcommand{\nid}{\noindent}

\newcommand{\iny}{\infty}

\newcommand{\norm}[1]{\left\vert \left\vert #1\right\vert\right\vert}
\newcommand{\innp}[1]{\left< #1 \right>}
\newcommand{\abs}[1]{\left\vert#1\right\vert}
\newcommand{\set}[1]{\left\{#1\right\}}

\newcommand{\pr}[1]{\left( #1 \right) }
\newcommand{\su}{\subset}

\newcommand{\bu}{\bigcup}

\newcommand{\lra}{\longrightarrow}

\newcommand{\B}[1]{\ensuremath{\mathbf{#1}}}

\newcommand{\Cal}[1]{\ensuremath{\mathcal{#1}}}
\newcommand{\Fr}[1]{\ensuremath{\mathfrak{#1}}}

\newcommand{\N}{\ensuremath{\B{N}}}
\newcommand{\Q}{\ensuremath{\B{Q}}}
\newcommand{\R}{\ensuremath{\B{R}}}
\newcommand{\Z}{\ensuremath{\B{Z}}}
\newcommand{\C}{\ensuremath{\B{C}}}
\newcommand{\X}{\mathfrak{X}}

\newcommand{\ab}{\mathbf{A}}

\newcommand{\id}{\mathbf{I}}

\usepackage{fancyhdr}

\fancypagestyle{plain}

\begin{document}


\title{\textbf{Decision problems, complexity, \\ traces, and representations}}
\author{Sean Lawton, Larsen Louder, D.~B.~McReynolds}
\maketitle


\begin{abstract}
\nid In this article, we study connections between representation theory and efficient solutions to the conjugacy problem on finitely generated groups. The main focus is on the conjugacy problem in conjugacy separable groups, where we measure efficiency in terms of the size of the quotients required to distinguish a distinct pair of conjugacy classes. 
\end{abstract}

\section{Introduction}

Given an infinite, finitely presented group $\Gamma$, two basic decision problems posed by Dehn \cite{Dehn} in 1911 are the word and conjugacy problems. In 1927, in solving the word problem for free groups, Schreier \cite{Schreier} proved that free groups are residual finite. That seems to be the first  connection between decision problems and residual properties. In 1940, Mal'cev \cite{MalcevRF} proved that finitely presented, residually finite groups have a solution to the word problem, and noted a similar connection between the conjugacy problem and conjugacy separability in \cite{MalcevCS}. 

Once an algorithm to solve the word or conjugacy problem is given, one can study the efficiency of the algorithm. For free groups, it is straightforward to see that both problems have algorithms that terminate in linear steps as a function of word length via cyclic reduction. Bou-Rabee \cite{BouRabee10} introduced a function $\WDep_\Gamma(n)$ that quantified the efficiency of the solution to the word problem on $\Gamma$ given by residual finiteness. We say a group $\Gamma$ is \textbf{residually finite} if for each $\gamma \in \Gamma$ with $\gamma \ne 1$, there exists a homomorphism $\vp\colon \Gamma \to Q$ with $\abs{Q} < \iny$ and $\vp(\gamma) \ne 1$. The function introduced by Bou-Rabee measures the efficiency by the size of the groups $Q$ needed over all the elements of length at most $n$ in the verification of residual finiteness. Several papers have addressed the growth rate of this function for various classes of groups; \cite{BouRabee10}, \cite{BouRabee11}, \cite{BHP}, \cite{BK}, \cite{BM10}, \cite{BM11}, \cite{BM12}, \cite{BS13}, \cite{Buskin}, \cite{KM}, \cite{KMS}, \cite{KT}, \cite{Kuperberg}, \cite{Patel}, \cite{Patel2}, \cite{Rivin}, and \cite{Thom}. By work of Mal'cev \cite{MalcevRF}, a finitely generated linear group $\Gamma$ is residually finite. In \cite{BM12}, using an effective proof of \cite{MalcevRF}, it was shown that $\WDep_\Gamma(n) \preceq n^d$ where $d$ depends only on a linear realization of $\Gamma$.  

We say that $\Gamma$ is \textbf{conjugacy separable} if for any non-conjugate $\gamma,\eta \in\Gamma$, there exists a homomorphism $\vp\colon \Gamma \to Q$ with $\vp(\gamma),\vp(\eta)$ not conjugate in $Q$ and $\abs{Q}<\iny$. One of the goals of this article is to extend some of the above results with conjugacy separability in place of residual finiteness. Unfortunately, issues arise immediately. Stebe \cite{Stebe} proved that the linear groups $\SL(n,\Z)$ are not conjugacy separable for $n>2$. More generally, the groups of integer points of a semi-simple $\Q$--algebraic group with the congruence subgroup property are never conjugacy separability; see \cite[Ch~8]{PLR}. However, free and surface groups \cite{Stebe,LS} (see also \cite{Paris} and \cite{Wilton}), virtual polycyclic groups \cite{Formanek,Rem}, fundamental groups of compact, orientable 3--manifolds \cite{HWZ}, and right-angled Artin groups \cite{Toi} are conjugacy separable; see also \cite{CBW} for more examples. 
 
A faithful linear representation reduces the verification of the non-triviality of an element to showing some matrix coefficient is non-zero. We want a similar solution to the conjugacy problem through representation theory and must replace the coefficients of the matrix by conjugacy invariants. We use traces and the following properties to effectively distinguish conjugacy classes:

\begin{itemize}
\item[(A)]
There exists an integer $n$ and $\rho\in \Hom(\Gamma,\SL(n,\C))$ such that $\Tr(\rho(\gamma)) \ne \Tr(\rho(\eta))$ for any non-conjugate $\gamma,\eta \in \Gamma$.
\item[(B)]
For each $\gamma \in \Gamma$, there exists $\rho_\gamma\in\Hom(\Gamma,\SL(n_\gamma,\C))$ such that $\Tr(\rho_\gamma(\gamma)) \ne \Tr(\rho_\gamma(\eta))$ for every non-conjugate $\eta \in \Gamma$.
\item[(C)]
For any finite set $S = \set{\gamma_i}_{i=1}^s$ of conjugacy classes in $\Gamma$, there exists $\rho_S\in\Hom(\Gamma,\SL(n_S,\C))$ such that $\Tr(\rho_S(\gamma_i)) \ne \Tr(\rho_S(\gamma_j))$ for $\gamma_i,\gamma_j \in S$ and $i \ne j$. 
\item[(D)]
For each non-conjugate $\gamma,\eta \in \Gamma$, there exists $\rho_{\gamma,\eta}\in\Hom(\Gamma,\SL(n_{\gamma,\eta},\C))$ such that $\Tr(\rho_{\gamma,\eta}(\gamma)) \ne \Tr(\rho_{\gamma,\eta}(\eta))$.
\end{itemize}

We have $\xymatrix{ \textrm{(A)} \ar@{=>}[r] & \textrm{(B)} \ar@{=>}[r] & \textrm{(D)} \ar@{<=>}[r] & \textrm{(C).} }$ All of these implications are immediate from the definitions except for the equivalence of (C) and (D), which is elementary. We thank Greg Kuperberg for pointing that out to us. We say one of the above (B), (C), or (D) is \textbf{uniformly satisfied} if $n_\gamma,n_{\gamma,\eta}$, or $n_S$ is bounded over all choices of $\gamma$, $\set{\gamma,\eta}$, or $S$. That is, the dimension of the representations do not depend on $\gamma$, $\set{\gamma,\eta}$, or $S$. In those cases, we say $\Gamma$ \textbf{uniformly} has (B), (C), or (D). Note, it is less clear if uniform (C) and uniform (D) are equivalent.  

\begin{rem}\label{prime}
Since $\SL_n$ and consequently $\Hom(\Gamma, \SL_n)$ are $\Z$-schemes, the above properties (A)-(D) can be restated with $\C$ replaced by any algebraically closed field $\mathbf{F}$.  When we are not working over $\C$ we will refer to these properties as (A$^\prime$)-(D$^\prime$).   For example, with respect to a fixed algebraically closed field $\mathbf{F}$, (A$^\prime$) states there exists an integer $n$, and $\rho\in \Hom(\Gamma,\SL(n,\mathbf{F}))$ such that $\Tr(\rho(\gamma)) \ne \Tr(\rho(\eta))$ for any non-conjugate $\gamma,\eta \in \Gamma$.  Properties (B$^\prime$)-(D$^\prime$) are similarly written.
\end{rem}

\begin{thm}\label{UniformCMeansA}
If $\Gamma$ uniformly has (C), then $\Gamma$ has (A). In fact, if $\Gamma$ uniformly has (D) for some $n_0$ and $\Hom(\Gamma,\SL(n_0,\C))$ is irreducible, then $\Gamma$ has (A).
\end{thm}

Throughout, by a surface group, we mean the fundamental group of a closed, orientable surface of genus $g \geq 2$. We have the following corollary:

\begin{cor}\label{Baire}
If $\Gamma$ is either a finitely generated free group or a surface group, then $\Gamma$ uniformly has (D) if and only if $\Gamma$ has  (A). Moreover, for any connected algebraic subgroup $\B{G}<\SL(n,\C)$, the following are equivalent for a free group $F_r$ of rank $r$:
\begin{itemize}
\item[(a)]
For each $\rho\in\Hom(F_r,\B{G})$, there exist non-conjugate $\gamma,\eta \in F_r$ with $\Tr(\rho(\gamma))= \Tr(\rho(\eta))$.
\item[(b)]
There exist non-conjugate $\gamma,\eta \in \Gamma$ such that $\Tr(\rho(\gamma))= \Tr(\rho(\eta))$ for each $\rho\in\Hom(F_r,\B{G})$.
\end{itemize}
\end{cor}

We also record the following result which first appeared in Bass--Lubotzky \cite[Prop.~3.1]{BL} where they also prove the converse.

\begin{prop}[Bass--Lubotzky]\label{PropCImpliesCS}
If $\Gamma$ satisfies (D), then $\Gamma$ is conjugacy separable. 
\end{prop}

Similar to the function $\WDep_\Gamma(n)$ associated to the word problem using residual finiteness, we define a function $\Conj_\Gamma(n)$ for the conjugacy problem using conjugacy separability (see \S \ref{Prelim} for the definition). 

\begin{thm}\label{UniformUniformComplexity}
If $\Gamma$ has (A), then $\Conj_\Gamma(n) \preceq n^d$ for some $d \in \N$. Moreover, for some $n_0 \in \N$, the finite quotients used in proving conjugacy separability of $\Gamma$ are subgroups of the finite groups $\SL(n_0,\B{F}_p)$ where $\B{F}_p$ denotes a field of prime order $p$.
\end{thm}

We define a relative version of the function $\Conj_\Gamma(n)$ by fixing a conjugacy class $[\gamma]$ in $\Gamma$ and denote this function by $\Conj_{\Gamma,\gamma}(n)$. The analog of Theorem \ref{UniformUniformComplexity} holds with (B) and $\Conj_{\Gamma,\gamma}(n)$ in place of (A) and $\Conj_\Gamma(n)$.

\begin{thm}\label{UniformComplexity}
If $\Gamma$ has (B), then for each $\gamma \in \Gamma$, there exists $d_\gamma \in \N$ such that $\Conj_{\Gamma,\gamma}(n) \preceq n^{d_\gamma}$.
\end{thm}

\paragraph{Property (A).} We now address the likelihood a group satisfies (A) or (B). We begin with (A). The obvious test case to begin investigating with regard to property (A) is finitely generated free groups. For $n=2$, Horowitz \cite{Horowitz} proved that there exist non-conjugate $\gamma,\eta \in F_2$ such that for any representation $\rho\colon F_2 \to \SL(2,\C)$, we have $\Tr(\rho(\gamma)) = \Tr(\rho(\eta))$. We say such elements are \textbf{$\SL_2$--trace equivalent}. It seems to have been, for some time now, a folklore question as to whether or not there exists $\SL_n$--trace equivalent elements in $F_2$ for $n>2$. In Section \ref{Horowitz}, we discuss whether or not the elements constructed by Horowitz can be $\SL_n$--trace equivalent, and see that if they are, an unexpected trace relation must hold. Ginzburg--Rudnick \cite{GR} investigated when a given element has a $\SL_2$--trace companion and gave a conjectural condition on the element to ensure that it does not have such a companion. Anderson \cite{Anderson} provided a broader context for the construction of Horowitz and a conjectural picture for what such pairs of $\SL_2$--trace equivalent elements should look like. Additionally, Leininger \cite{Leininger} and Kapovich--Levitt--Schupp--Shpilrain \cite{KLSS} gave a more geometric/topological take (see also \cite{GR}, \cite{Lee}, \cite{Lee2}, \cite{LV}). Of course, we have trivially that any $\SL_3$--trace equivalent pair is also an $\SL_2$--trace equivalent pair. The failure of Anderson's general construction to produce $\SL_3$--trace equivalent pairs would provide some evidence that free groups have (A).

The most compelling evidence against free groups having (A) is Theorem \ref{UniformUniformComplexity}. By \cite{Thom} and \cite{BouRabee10}, the function $\WDep_{F_r}(n)$ satisfies $n(\log\log(n))^{9/2}/(\log(n))^2 \preceq \WDep_{F_r}(n) \preceq n^3$. We believe that the growth rate of $\Conj_{F_r}(n)$ should be much greater since conjugacy separability requires separating a fixed element $\gamma$ from an infinite set while residual finiteness requires only separating $\gamma$ from the trivial element. It is this reason why many linear groups are not conjugacy separable. However, if $F_r$ has (A), then by Theorem \ref{UniformUniformComplexity}, we would have, for some fixed $d$, the asymptotic inequalities $\Conj_{F_r}(n) \preceq n^d$. In particular, $\Conj_{F_r}(n) \preceq (\WDep_{F_r}(n))^{3d}$. For any finitely generated abelian group, these two functions are the same, and the best setting to hope for a power relationship like $\Conj_\Gamma(n) \preceq (\WDep_\Gamma(n))^d$ is the class of torsion free, finitely generated nilpotent groups where conjugacy classes are relatively small. However, by \cite{BouRabee10} and \cite{Peng}, a torsion free, finitely generated nilpotent group satisfies such a power relationship if and only if the group is virtually abelian. 

\paragraph{Property (B).} 

Following a construction of Wehrfritz \cite{Wehrfritz} for free groups, we can prove that finitely generated free groups and surface groups have (B).

\begin{thm}\label{FreeHaveB}
If $\Gamma$ is a finitely generated free group or surface group, then $\Gamma$ has (B).
\end{thm}

From Theorem \ref{UniformComplexity} and Theorem \ref{FreeHaveB}, we obtain:

\begin{cor}\label{FreeRelComplexity}
If $\Gamma$ is a finitely generated free group or surface group and $\gamma \in \Gamma$, then there exists $d_\gamma \in \N$ such that $\Conj_{\Gamma,\gamma}(n) \preceq n^{d_\gamma}$. Moreover, one can take $d_\gamma \approx \norm{\gamma}^2$ and thus $\Conj_\Gamma(n) \preceq n^{n^2}$.
\end{cor}

The degree $d_\gamma$ in Corollary \ref{FreeRelComplexity} is directly related to the smallest index of a finite index subgroup $\Gamma$ of $F_r$ where $\gamma \in \Gamma$ is primitive. In the case of surface groups, it is directly related to the smallest degree of a cover where the curve associated to $\gamma$ has a lift to a simple closed curve. Patel \cite{Patel} and Gupta--Kapovich \cite{GK} have given upper bounds in the case of surface groups and free groups, respectively, on order $\norm{\gamma}$. Gaster \cite{Gas}, improving on work of Gupta--Kapovich \cite{GK}, has shown that there exist $\gamma$ that require a cover of degree on the order of $\norm{\gamma}$. We conjecture that there is no polynomial upper bound for $\Conj_{F_r}(n)$, and coupled with Theorem \ref{UniformUniformComplexity}, that would imply that free groups do not have (A). 

\begin{conj}
Finitely generated free groups do not have (A).
\end{conj}

Finally, we prove a result that shows that for fully residually free groups, one can recover the profinite topology via the topology generated by the Zariski topologies for faithful representations into $\SL(n,\C)$ as we vary over all $n \in \N$. Recall that $\Gamma$ is \textbf{fully residually free} if for each finite subset $S \su \Gamma$ of non-trivial elements, there exists $r_S \in \N$ and a homomorphism $\psi_S\colon \Gamma \to F_{r_S}$ such that the restriction of $\psi_S$ to $S$ is injective. Examples of fully residually free groups are free groups and surface groups.

\begin{thm}\label{RepresentationTopology}
Let $\Gamma$ be a fully residually free group, $\Delta$ a finite index, normal subgroup of $\Gamma$, and $p \in \N$ a prime. Then there exists an integral domain $R \su \C$, an ideal $\mathfrak{m} \su R$ with $R/\mathfrak{m}=\mathbf{F}_p$, and a faithful homomorphism $\rho\colon \Gamma \to \SL(n_\Delta,R)$ such that $\Delta = \ker(r_{\mathfrak{m}} \circ \rho)$ where $r_{\mathfrak{m}}\colon \SL(n_\Delta,R) \to \SL(n_\Delta,\mathbf{F}_p)$ is the reduction modulo $\mathfrak{m}$ homomorphism and $n_\Delta = 2[\Gamma:\Delta]$.
\end{thm}

The ring $R$ can be taken to be finitely generated over $\Z$ (see Remark \ref{R:Ring}), and when $\Gamma$ is a free group, we can take $R=\Z$ (see Remark \ref{R:Free}). When $\Gamma$ is an arithmetic lattice in a $\Q$--algebraic group $\B{G}$, the congruence subgroup property asserts that every finite index subgroup $\Lambda < \Gamma$ contains $\ker(r_m)$ for some integer $m \in \N$. Every non-abelian free group $F_r$ can be realized as a finite index subgroup of $\SL(2,\Z)$ and it is well-known that $\SL(2,\Z)$ does not have the congruence subgroup property. The above result provides a weaker property than the congruence subgroup property when $\Gamma$ is a limit group. As we mentioned above, lattices in semi-simple Lie groups with the congruence subgroup property are not conjugacy separable and so do not have (D). These groups are super-rigid and the Zariski topology associated to the standard representation, which is the congruence topology, is too coarse for separating conjugacy classes. That free groups and surface groups are conjugacy separable is due to their much richer representation theory. We believe conjugacy separability requires linear representations of arbitrarily large dimension or finite quotients with arbitrarily large representation dimension. If Conjecture 1 is false, then free groups would be conjugacy separable via the Zariski topology associated to a fixed finite dimensional representation. In fact, for sufficiently large $n$ and a generic (in the Baire Category sense) $\rho \in \Hom(F_r,\SL(n,\C))$, every conjugacy class in $F_r$ would be closed in the Zariski topology associated to $\rho$.

\paragraph{Acknowledgements.}

The authors would like to thank Nigel Boston, Khalid Bou-Rabee, Frank Calegari, Ted Chinburg, Kelly Delp, Nathan Dunfield, Patrick Eberlin, Carolyn Gordon, Adrian Ioana, Mike Jablonski, Ilya Kapovich, Greg Kuperberg, Chris Leininger, Darren Long, Alex Lubotzky, Dave Morris, Alan Reid, Mark Sapir, Juan Souto, and Pete Storm for stimulating conversations on the topics of this article. Special thanks is given to Dunfield who many years ago posed the question to the third author on the possibility of generalizing Horowitz's construction. The first author was partially supported by NSF grant 1309376 and Simons grant 245642. The third author was partially supported by NSF grant 1105710. Lastly, we want to thank an anonymous referee for very good remarks that helped improve the paper.

\section{Preliminaries}\label{Prelim}

\subsection{Quantitative separability functions}

Given $f,g\colon \N \to \N$, we say $f \preceq g$,  if $f(n) \leq Cg(Cn)$ for some constant $\C\in \N$ and for all $n \in \N$. If $f \preceq g$ and $g \preceq f$, we write $f \approx g$. Throughout, $\Gamma$ will denote an infinite, finitely generated group unless stated otherwise. Given a finite generating set $X$ of $\Gamma$ and $\gamma \in \Gamma$, we denote the word length of $\gamma$ with respect to $X$ by $\norm{\gamma}_X$ (or simply $\norm{\gamma}$) and $n$--ball with respect to the associated word metric by $B_{\Gamma\!\!,\ X}(n)$. Given $\Gamma$, we define $\D_\Gamma\colon \Gamma-\set{1} \to \B{N} \cup \set{\iny}$ by 
\[ \D_\Gamma(\gamma) = \min\set{[\Gamma:\Delta]~:~\gamma \notin \Delta,~\Delta \lhd \Gamma} \] 
and $\WDep_{\Gamma,X}(n)$ by
\[ \WDep_{\Gamma,X}(n) = \max_{\gamma \in B_{\Gamma\!\!,\ X}(n) - \set{1}} \D_\Gamma(\gamma). \] 
For any two finite generating sets $X,Y$, we have $\WDep_{\Gamma,X} \approx \WDep_{\Gamma,Y}$ (see \cite[Lem~1.1]{BouRabee10}). Consequently, we suppress the dependence on $X$ in our notation. For a finitely generated group $\Gamma$ and $\gamma \in \Gamma$, we denote the $\Gamma$--conjugacy class of $\gamma$ by $[\gamma]$ and denote the set of $\Gamma$--conjugacy classes by $C_\Gamma$. For $[\gamma] \in C_\Gamma$, we define $\norm{~[\gamma]~}_X = \min\set{\norm{\gamma'}_X~:~\gamma' \in [\gamma]}$, and $\Con_\Gamma\colon C_\Gamma \times C_\Gamma \to \N \cup \set{\iny}$ by
\[ \Con_\Gamma([\gamma],[\eta]) = \min\set{\abs{Q}~:\ \vp\colon \Gamma \to Q,~[\vp(\gamma)]_Q \ne [\vp(\eta)]_Q}. \]
By definition, for $\gamma,\eta \in \Gamma$ with $[\gamma] \ne [\eta]$, we have
\[ \Con_\Gamma([\gamma],[\eta]) \geq \max\set{\D_\Gamma(\gamma^{-1}\eta')~:~\eta' \in [\eta]}. \]
We define $B_X(C_\Gamma,n) = \set{[\gamma]~:~\norm{~[\gamma]~}_X \leq n}$ and $\Conj_{\Gamma,X} \colon \N \lra \N \cup \set{\iny}$ via
\[ \Conj_{\Gamma,X}(n) = \max_{[\gamma],[\eta] \in B_X(C_\Gamma,n),~[\gamma] \ne [\eta]} \Con_\Gamma([\gamma],[\eta]). \]
For $[\gamma] \in C_\Gamma$, we define $\Con_{\Gamma,\gamma}\colon C_\Gamma - \set{[\gamma]} \to \N \cup \set{\iny}$ to be $\Con_{\Gamma,\gamma}([\eta]) = \Con_\Gamma([\gamma],[\eta])$ and 
\[ \Conj_{\Gamma,\gamma,X}(n) = \max_{[\eta] \in B_X(C_\Gamma,n),~[\gamma] \ne [\eta]} \Con_{\Gamma,\gamma}([\eta]). \]
For any two finite generating sets $X,Y$ of $\Gamma$, we have $\Conj_{\Gamma,X}(n) \approx \Conj_{\Gamma,Y}(n)$ and $\Conj_{\Gamma,\gamma,X}(n) \approx \Conj_{\Gamma,\gamma,Y}(n)$. The proof is similar to the proof of the comparable statement for the function $\WDep_\Gamma(n)$; see \cite[Lem~1.1]{BouRabee10}. As a result, we suppress the dependence on the generating set $X$ in our notation.

\subsection{Representation varieties}

We refer the reader to \cite[\S 5]{BGSS}, \cite[\S 2]{Goldman}, and \cite[Ch.~V]{Rag} for the material in this subsection. If $G$ is a Lie (resp. algebraic) group and $\Gamma = F_r$, then $\Hom(F_r,G) = G^r$ is an analytic (resp. algebraic) variety. More generally, when $\Gamma$ is finitely generated, $\Hom(\Gamma,G)$ will be an analytic (resp. algebraic) subvariety of $\Hom(F_r,G)$ for some $r$; see \cite[\S 5]{BGSS}. For each $\gamma \in \Gamma$, we have an analytic (resp. algebraic) function $\Hom(\Gamma,G) \to G$ given by $\rho \mapsto \rho(\gamma)$. If $G < \GL(n,\C)$, the function $\Tr_\gamma\colon \Hom(\Gamma,G) \to \C$ given by $\Tr_\gamma(\rho) = \Tr(\rho(\gamma))$ is analytic (resp. algebraic). When $G$ is a $K$--algebraic group with $K$ a characteristic zero field, $\Hom(\Gamma,G)$ is a $K$--algebraic set (not necessarily irreducible or connected), and so has finitely many irreducible (and connected) components. In particular, for $G = \SL(n,\C)$, the space $\Hom(\Gamma, \SL(n,\C))$ is a complex algebraic variety with finitely many irreducible components. For a connected, reductive algebraic group $\B{G}$, the $\B{G}$--character variety $\X(\Gamma,\B{G})$ is the GIT quotient of $\Hom(\Gamma,\B{G})$ by the $\B{G}$--conjugation action, and for $\Gamma = F_r$, we set $\X_r(\B{G}) = \X(F_r,\B{G})$. Though we do not require it here, we include the following result on algebraic points of character varieties that we could not find explicitly in the literature (it is implicit in \cite[Prop~6.6]{Rag}). 

\begin{thm}
If $\B{G}$ is a connected, reductive, affine algebraic group, then $\Hom(\Gamma,\B{G}(\overline{\Q}))$ is classically dense in $\Hom(\Gamma,\B{G}(\C))$, and $\X(\Gamma, \B{G}(\overline{\Q}))$ is classically dense in $\X(\Gamma, \B{G}(\C))$.
\end{thm}

\begin{proof}
First note that for any $d$--dimensional affine variety $V$ defined over $\Q$ the Noether normalization map $V\to \mathbf{A}^d$ is surjective and defines a finite cover off its branch locus.  Since the branch locus is nowhere dense, the $\overline{\Q}$--points are both Zariski and classically dense in the $\C$--points of $V$. According to \cite[p.~220]{Borel}, $\B{G}(K)$ is Zariski dense in $\B{G}$ for any infinite subfield $K \subset \C$.  Since $\B{G}$ is defined over $\Q$ and the relations in $\Gamma$ are defined over $\Z$, $\Hom(\Gamma,\B{G})$ is an affine variety defined over $\Q$ and  
\begin{equation}\label{tensor}
\C[\Hom(\Gamma,\B{G})]=\Q[\Hom(\Gamma,\B{G})]\otimes_\Q \C. 
\end{equation} 
Hence, $\Hom(\Gamma,\B{G}(\overline{\Q}))$ is both Zariski and classically dense in $\Hom(\Gamma,\B{G}(\C))$.  Let $f_1,\dots,f_N$ be a set of generators for $\C[\Hom(\Gamma,\B{G})]^G$, and define  $f\colon \Hom(\Gamma,\B{G})\to \C^N$ by $$f(g_1,\dots,g_r)=(f_1(g_1,\dots,g_r),\dots,f_N(g_1,\dots,g_r)).$$  Since $\X(\Gamma, \B{G})=\mathrm{Spec}(\C[\Hom(\Gamma,\B{G})]^{\B{G}}),$ we have $\X(\Gamma, \B{G})=f(\Hom(\Gamma, \B{G}))$; see \cite{Schwarz} for example.  As $\C[\Hom(\Gamma,\B{G})]^{\B{G}} \subset \C[\Hom(\Gamma, \B{G})]$, Equation \eqref{tensor} implies that $f_1,\dots,f_N$ may be chosen to have $\Q$--coefficients. Thus, $f\pr{\Hom(\Gamma,\B{G}(\overline{\Q}))}\subset \X(\Gamma, G(\overline{\Q}))$. As $f$ is a continuous surjective function, we conclude that $f\pr{\Hom(\Gamma,\B{G}(\overline{\Q}))}$ is classically dense in $\X(\Gamma, \B{G}(\C))$. Hence, $\X(\Gamma, \B{G}(\overline{\Q}))$ is classically dense in $\X(\Gamma, \B{G}(\C))$ as it contains $f\pr{\Hom(\Gamma,\B{G}(\overline{\Q}))}$.
\end{proof}

\begin{cor}\label{int}
If $\B{G} =\SL(n,\C)$, the integral points are infinite in $\X_r(\B{G})$.
\end{cor}

\begin{proof}
For $\B{G} =\SL(n,\C)$, the group schemes and invariant rings in the above proof are defined over $\Z[1/n]$.  So, the result follows from the above proof noting that $\Hom(F_r, \B{G})\cong \B{G}^r$.
\end{proof}

From the work of Long and Reid \cite{LR}, one can infer that Corollary \ref{int}  is false for $\SL(2,\C)$ when $\Gamma$ is a surface group.

\section{Property (C): Proof of Theorem \ref{UniformCMeansA} and Proposition \ref{PropCImpliesCS}}

\subsection{Proof of Theorem \ref{UniformCMeansA} and Corollary \ref{Baire}}

We now prove that either uniform (C), or uniform (D) with the irreducibility of $\Hom(\Gamma,\SL(n,\C))$ imply property (A).

\begin{proof}[Proof of Theorem \ref{UniformCMeansA}]
We assume first that $\Gamma$ uniformly has (C). We enumerate the conjugacy classes of $\Gamma$ by $\set{[\gamma_1],[\gamma_2],\dots}$ and for each $j \in \N$, set $S_j = \set{[\gamma_i]}_{i=1}^j$. By assumption, there exists $n \in \N$ and for each $r$, we have a representation $\rho_r\colon \Gamma \to \SL(n,\C)$ such that $\Tr(\rho(\gamma_i)) \ne \Tr(\rho(\gamma_j))$ for all $i \ne j\leq r$. As $\Hom(\Gamma,\SL(n,\C))$ has only finitely many irreducible components, there exists a component that contains infinitely many of the representations $\rho_r$, say $V_0 \su \Hom(\Gamma,\SL(n,\C))$. By selection, the trace functions $\Tr_\gamma$ restricted to $V_0$ are distinct algebraic functions for each conjugacy class $[\gamma]$. In particular, $\Tr_{\gamma_i} - \Tr_{\gamma_j} \ne 0$ is a non-constant algebraic function on $V_0$ for each pair $i \ne j$.  In particular, the sets
\[ Z_{i,j} = \set{ \rho \in V_0~:~\Tr_{\gamma_i}(\rho) - \Tr_{\gamma_j}(\rho) = 0} \]
are proper algebraic subvarieties of $V_0$. By the Baire Category Theorem, $V = V_0 - \bigcup_{i,j} Z_{i,j}$ is dense and so non-empty. By construction, any $\rho \in V$ has the property that $\Tr(\rho(\gamma)) = \Tr(\rho(\eta))$ if and only if $\gamma,\eta$ are conjugate in $\Gamma$. In particular, $\Gamma$ has property (A).

In the case we uniformly have (D) and $\Hom(\Gamma,\SL(n,\C))$ is irreducible, we know that by assumption that for each pair of conjugacy classes $\gamma,\eta \in \Gamma$, we have a representation $\rho\colon \Gamma \to \SL(n,\C)$ with $\Tr(\rho(\gamma)) \ne \Tr(\rho(\eta))$. Since $\Hom(\Gamma,\SL(n,\C))$ is irreducible, we can proceed as before with $V_0 = \Hom(\Gamma,\SL(n,\C))$.
 \end{proof} 
 
 Before we prove Corollary \ref{Baire}, we note that in the special case of the genus $1$ surface, the fundamental group $\Z^2$ has (A). Take any two algebraically independent numbers $\alpha,\beta \in \R$. Fixing a $\Z$--basis $v,w$, we have the representation $\rho\colon \Z^2 \to \GL(2,\R)$ given by $\rho(av + bw) = \begin{pmatrix} \alpha^a & 0 \\ 0 & \beta^b \end{pmatrix}$. By selection of $\alpha,\beta$, distinct elements in $\Z^2$ will have distinct traces. The groups $\Z^n$ also have (A) for any $n \in \N$.
 
\begin{proof}[Proof of Corollary \ref{Baire}]
The first part of Corollary \ref{Baire} follows immediately from the irreducibility of $\Hom(F_r,G) = G^r$ for any connected algebraic group over $\C$ in the case of free groups. For a closed, orientable surface $\Sigma_g$ of genus $g \geq 2$, $\Hom(\pi_1(\Sigma_g),\SL(n,\C))$ is irreducible by \cite{RBC} and \cite[Lem,~2.5]{BLR} (the same holds for $g=1$; see \cite[Prop.~5.16]{FL}). For the second part, we must prove that the following two statements are equivalent:
\begin{itemize}
\item[(a)]
There exists non-conjugate $\gamma,\eta \in F_r$ that have $\Tr(\rho(\gamma)) = \Tr(\rho(\eta))$ for every $\rho\in \Hom(F_r,G)$.
\item[(b)]
For each $\rho\in \Hom(F_r,G)$, there exist non-conjugate $\gamma,\eta \in F_r$ such that $\Tr(\rho(\gamma)) = \Tr(\rho(\eta))$.
\end{itemize}

It is clear that (a) implies (b). To prove that (b) implies (a), we assume that (b) holds but not (a) and derive a contradiction. Since (a) does not hold, then for each non-conjugate pair $\gamma,\eta \in F_r$, the function $\Tr_\gamma - \Tr_\eta$ on $\Hom(F_r, G)$ is a non-constant algebraic function. Since $\Hom(F_r, G)$ is irreducible,
\[ V_{\gamma,\eta} = \set{\rho \in \Hom(F_r, G)~:~\Tr_\gamma(\rho) - \Tr_\eta(\rho) = 0} \]
is nowhere dense. Taking $V = \bu_{[\gamma] \ne [\eta]} V_{\gamma,\eta}$, by the Baire Category Theorem, $V$ is nowhere dense. Let $\rho \in \Hom(F_r, G) - V$ and note that by construction, no two non-conjugate elements have the same trace under $\rho$. That contradicts our assumption that (b) holds for every $\rho \in \Hom(F_r, G)$. 
\end{proof}

\subsection{Proof of Proposition \ref{PropCImpliesCS}}

The proof of Proposition \ref{PropCImpliesCS} is similar to Mal'cev's proof of residual finiteness for linear groups. As we will use some of the setup later, we give a proof here. A proof can also be found in \cite{BL}. 

\begin{proof}[Proof of Proposition \ref{PropCImpliesCS}]
Given non-conjugate $\gamma,\eta \in \Gamma$, we must find a homomorphism $\vp\colon \Gamma \to Q$ where $Q$ is a finite group such that $\vp(\gamma),\vp(\eta)$ are not conjugate in $Q$. By assumption, $\Gamma$ has property (D) and so there exists $\rho\in \Hom(\Gamma,\SL(n,\C))$ such that $\Tr(\rho(\gamma)) \ne \Tr(\rho(\eta))$. Since $\Gamma$ is finitely generated, the field $L$ generated over $\Q$ by the coefficients of the elements $\rho(\lambda)$ as we vary over all $\lambda \in \Gamma$ has the form $L = K(x_1,\dots,x_r)$, where $K/\Q$ is a finite extension and $x_1,\dots,x_r$ are indeterminants. It follows that $\rho(\Gamma) < \SL(n,R)$, where $R = S[1/\beta_1,\dots,1/\beta_t]$, $S = \Cal{O}_K[x_1,\dots,x_r]$, and $\Cal{O}_K$ is the ring of $K$--integers. We see then that $\Tr(\rho(\lambda)) \in R$ for each $\lambda \in \Gamma$. We know that $\Tr(\rho(\gamma)) - \Tr(\rho(\eta)) = F(x_1,\dots,x_{r'}) \in R$ is a non-zero polynomial in the variables $x_1,\dots,x_{r'}$ with coefficients in $S$. Since $F$ is non-zero, we can find $\alpha_1,\dots,\alpha_{r'} \in S$ such that $\alpha = F(\alpha_1,\dots,\alpha_{r'}) \ne 0$ with $\alpha \in S$. As there are only finitely many prime ideals $\Fr{p}$ in $S$ such that $\alpha = 0\mod \Fr{p}$, we select a prime $\Fr{p}$ for which $\alpha \ne 0 \mod \Fr{p}$. For such a prime, the ring homomorphisms $R \to S \to S/\Fr{p} \cong \B{F}_q$ induce homomorphisms $\Gamma \to \SL(n,R) \to \SL(n,S) \to \SL(n,\B{F}_q)$. Set $\vp\colon \Gamma \to \SL(n,\B{F}_q)$ to be the resulting map. By construction $\Tr(\vp(\gamma)) \ne \Tr(\vp(\eta))$ and so $\vp(\gamma),\vp(\eta)$ are not conjugate in $\SL(n,\B{F}_q)$.
\end{proof}

\subsection{Ultraproducts}

For a fixed $n \in \N$, we say that a group $\Gamma$ is \textbf{$n$--trace distinguished} if for each non-conjugate pair $\gamma,\eta \in \Gamma$, there exists a finite field $\B{F}_q$ and a homomorphism $\varphi\colon \Gamma \to \SL(n,\B{F}_q)$ such that $\Tr(\vp(\gamma)) \ne \Tr(\vp(\eta))$. We say $\Gamma$ is \textbf{fully $n$--trace distinguished} if for any finite set $S = \set{\gamma_j}_{j=1}^s \su \Gamma$ of pairwise non-conjugate elements, there exists a finite field $\B{F}_q$ and a homomorphism $\vp\colon \Gamma \to \SL(n,\B{F}_q)$ such that $\Tr(\vp(\gamma_i)) \ne \Tr(\vp(\gamma_j))$ for all $1 \leq i < j \leq s$.

\begin{thm}\label{T:UltraA}
If $\Gamma$ is finitely generated and fully $n$--trace distinguished for some $n \in \N$, then $\Gamma$ has (A).
\end{thm}

In the proof of Theorem \ref{T:UltraA}, we employ ultraproducts. We refer the reader to \cite{Hal95} for an introduction to these methods.

\begin{proof}
To begin, we enumerate the conjugacy classes of $\Gamma$ by $\set{[\gamma_1],[\gamma_2],\dots}$ and for each $j \in \N$, set $S_j = \set{[\gamma_i]}_{i=1}^j$. By assumption, for each $j \in \N$, there exists a finite field $\B{F}_{q_j}$ and a homomorphism $\vp_j\colon \Gamma \to \SL(n,\B{F}_{q_j})$ such that $\Tr(\vp_j(\gamma_i)) \ne \Tr(\vp_j(\gamma_{i'}))$ for all $1 \leq i < i' \leq j$. Picking a non-principal ultrafilter $\omega$ on $\N$, the ultraproduct $\prod_\omega \B{F}_{q_j} = \B{K}_\omega$ is a field and we have an induced homomorphism $\prod_\omega \vp_j = \Phi_\omega$, where $\Phi_\omega\colon \Gamma \to \SL(n,\B{K}_\omega)$. By selection of the homomorphisms $\vp_j$, it follows that $\Tr(\Phi_\omega(\gamma_i)) \ne \Tr(\Phi_\omega(\gamma_{i'})$ for all $i \ne i'$. Hence, $\Gamma$ has (A).
\end{proof}

The field $\B{K}_\omega$ may have positive characteristic and so in the definition of (A), we must allow for algebraically closed fields of positive characteristic (see Remark \ref{prime}). Using the methods from the proof of Proposition \ref{PropCImpliesCS}, it is straightforward to see that if $\Gamma$ has (A$^\prime$), then $\Gamma$ is fully $n$--trace distinguished. 

We can also consider a relative version of $n$--trace distinguished. For $\gamma \in \Gamma$ and $n \in \N$, we say $\gamma$ is \textbf{$n$--trace distinguished in $\Gamma$} if for each non-conjugate $\eta \in \Gamma$, there exists a finite field $\B{F}_q$ and a homomorphism $\vp\colon \Gamma \to \SL(n,\B{F}_q)$ such that $\Tr(\vp(\gamma)) \ne \Tr(\vp(\eta))$. We say $\gamma$ is \textbf{fully $n$--trace distinguished in $\Gamma$} if for any finite set $S = \set{\gamma_j}_{j=1}^s \su \Gamma$, none of which is conjugate to $\gamma$, there exists a finite field $\B{F}_q$ and a homomorphism $\vp\colon \Gamma \to \SL(n,\B{F}_q)$ such that $\Tr(\vp(\gamma)) \ne \Tr(\vp(\gamma_j))$ for all $1 \leq j  \leq s$.

\begin{thm}\label{T:UltraB}
If $\Gamma$ is finitely generated and for each $\gamma \in \Gamma$, there exists $n_\gamma \in \N$ such that $\gamma$ is fully $n_\gamma$--trace distinguished, then $\Gamma$ has (B).
\end{thm}

\begin{proof}
To begin, we enumerate the conjugacy classes of $\Gamma$ by $\set{[\gamma_1] = [\gamma], [\gamma_2], [\gamma_3],\dots}$ and for each $j \in \N$, set $S_j = \set{[\gamma_i]}_{i=2}^j$. By assumption, for each $j \geq 2$, there exists a finite field $\B{F}_{q_j}$ and a homomorphism $\vp_j\colon \Gamma \to \SL(n,\B{F}_{q_j})$ such that $\Tr(\vp_j(\gamma_i)) \ne \Tr(\vp_j(\gamma))$ for all $2 \leq i \leq j$. Picking a non-principal ultrafilter $\omega$ on $\N$, the ultraproduct $\prod_\omega \B{F}_{q_j} = \B{K}_\omega$ is a field and we have an induced homomorphism $\prod_\omega \vp_j = \Phi_\omega$, where $\Phi_\omega\colon \Gamma \to \SL(n,\B{K}_\omega)$. By selection of the homomorphisms $\vp_j$, it follows that $\Tr(\Phi_\omega(\gamma_i)) \ne \Tr(\Phi_\omega(\gamma))$ for all $i \geq 2$. Hence, $\Gamma$ has (B).
\end{proof}

As before, allowing for algebraically closed fields of positive characteristic in our definition of (B), the converse holds assuming (B$^\prime$). 

\subsection{Proof of Theorem \ref{UniformUniformComplexity} and Theorem \ref{UniformComplexity}}

We now prove Theorem \ref{UniformUniformComplexity}. 

\begin{proof}[Proof of Theorem \ref{UniformUniformComplexity}]
 We assume that $\Gamma$ has (A) for some integer $m \in \N$, and so there exists $\rho\in\Hom(\Gamma,\SL(m,\C))$ such that $\Tr(\rho(\gamma)) \ne \Tr(\rho(\eta))$ for any non-conjugate $\gamma,\eta \in \Gamma$. For simplicity, we assume that $\rho(\Gamma) < \SL(m,\overline{\Q})$, as the alternative $\rho(\Gamma) < \SL(m,K[x_1,\dots,x_r])$, where $K/\Q$ is a finite extension, is handled similarly (see \cite{BM12}). We must prove that for any non-conjugate pair $\gamma,\eta \in \Gamma$ with $\norm{\gamma},\norm{\eta} \leq n$, that $\Con_\Gamma(\gamma,\eta) \leq Cn^{m^2-1}$ for a constant $C$ that is independent of $\gamma,\eta$. To begin, we can find a finite extensions $K/\Q$ and $S/\Cal{O}_K$ such that $\rho(\Gamma) < \SL(m,S)$. With this setup, we know for any non-conjugate $\gamma,\eta$ that $\Tr(\rho(\gamma)) - \Tr(\rho(\eta)) \in S$ and also is non-zero. We need an ideal $\Fr{a}$ of $S$ such that $\Tr(\rho(\gamma)) - \Tr(\rho(\eta)) \ne 0 \mod \Fr{a}$ and with $\abs{S/\Fr{a}}$ small. We achieve this goal using the methods of \cite{BouRabee10} (or \cite{BM12}). First, we control the size of the coefficients of $\rho(\gamma),\rho(\eta)$ as a function of word length. To that end, it follows (see \cite{BouRabee10} or \cite{BM12}) that there exists constants $\alpha$ and $C_0$ depending only on the generators of $\Gamma$ such that
\[ \max\set{\abs{(\rho(\gamma))_{i,j}}~:~i,j \in \set{1,\dots,m}} \leq \alpha^{C_0\norm{\gamma}}. \]
In particular, given non-conjugate $\gamma,\eta \in \Gamma$ with $\norm{\gamma},\norm{\eta} \leq n$, we see that
\[ \abs{\Tr(\rho(\gamma)) - \Tr(\rho(\eta))} \leq \abs{\Tr(\rho(\gamma))} + \abs{\Tr(\rho(\eta))} \leq 2m\alpha^{C_0n}. \]
By \cite[Thm~2.4]{BouRabee10}), we can find a prime ideal $\Fr{p}$ with 
\[ \abs{S/\Fr{p}} \leq C_1\log(C_12m\alpha^{C_0n}) \leq  C_1C_0n\log(C_12m\alpha) \] 
such that $\Tr(\rho(\gamma)) \ne \Tr(\rho(\eta)) \mod \Fr{p}$. The constant $C_1$ depends only on the ring $S$. Let $r_\Fr{p}\colon \SL(n,S) \to \SL(n,S/\Fr{p})$ be the reduction modulo $\Fr{p}$ homomorphism and set $\rho_\Fr{p}\colon \Gamma \to \SL(n,S/\Fr{p})$ by $\rho_\Fr{p} = r_\Fr{p} \circ \rho$. By selection of $\Fr{p}$, we see that $\rho_\Fr{p}(\gamma),\rho_\Fr{p}(\eta)$ have distinct traces and hence have non-conjugate images. We also have
\[ \abs{\rho_\Fr{p}(\Gamma)} \leq \abs{\SL(n,S/\Fr{p})} \leq \abs{S/\Fr{p}}^{m^2-1} \leq (C_1n\log (C_12m\alpha))^{m^2-1} = Cn^{m^2-1} \]
where $C$ is the constant $(C_1C_0\log(C_12m\alpha))^{m^2-1}$. In particular, $\Con_\Gamma([\gamma],[\eta]) \leq Cn^{m^2-1}$ for some constant $C$ depending only on $\Gamma$ and $\rho$. As this holds for all $[\gamma],[\eta] \in B(C_\Gamma,n)$, we see that $\Conj_\Gamma(n) \preceq n^{m^2-1}$. The assertion that one only needs subgroups of $\SL(n_0,\B{F}_p)$ in proving conjugacy separability for $\Gamma$ follows from the \v{C}ebotarev Density Theorem.
\end{proof}

\begin{proof}[Proof of Theorem \ref{UniformComplexity}]
We proceed similarly to the proof of Theorem \ref{UniformUniformComplexity}. By assumption, we have $\rho \in \Hom(\Gamma,\SL(n_\gamma,\C))$ such that $\Tr(\rho(\gamma)) \ne \Tr(\rho(\eta))$ for any $\eta \in \Gamma$ that is not conjugate to $\gamma$. Using $\Tr(\rho(\gamma)) - \Tr(\rho(\eta))$, we can employ the same methods used in the proof of Theorem \ref{UniformUniformComplexity} to find the desired homomorphism to a finite group where $\gamma,\eta$ have non-conjugate images.
\end{proof}
 
\section{Horowitz's construction}\label{Horowitz}

In this section we show that the cyclically reduced words constructed in Example 8.2 in \cite{Horowitz} that do have the same trace over $\SL(2,\C)$ are not likely to have the same trace over $\SL(n,\C)$ for $n>2$.  Since $\SL(n-1,\C)$ embeds into $\SL(n,\C)$ it suffices to show that this failure occurs for $n=3$.

\subsection{Reduction to free groups}

The following lemma reduces the search for trace equivalent pairs in non-elementary hyperbolic groups to finding them in $F_r$

\begin{lemma}\label{L:Mal}
Let $n,r \geq 2$ be integers. If there exists a non-conjugate pair $w_1,w_2 \in F_r$ such that $w_1,w_2$ are $\SL_n$--trace equivalent, then for any non-elementary hyperbolic group $\Delta$, there exists non-conjugate $\delta_1,\delta_2 \in \Delta$ that are $\SL_n$--trace equivalent.
\end{lemma}

\begin{proof}
By I.~Kapovich \cite[Thm~C]{Kap}, $\Delta$ has a malnormal subgroup $\Delta_0$ that is isomorphic to $F_r$. Fixing any isomorphism $\psi\colon F_r \to \Delta_0$, we set $\delta_j = \psi(w_j)$. For any representation $\rho\colon \Delta \to \SL(n,\C)$, it follows that $\Tr(\rho(\delta_1)) = \Tr(\rho(\delta_2))$. As $\delta_1,\delta_2$ are non-conjugate in $\Delta_0$ and $\Delta_0$ is malnormal in $\Delta$, we see that $\delta_1,\delta_2$ are non-conjugate in $\Delta$.
\end{proof}

Since free groups are hyperbolic, it follows that for any integers $r,s \geq 2$, $F_r$ has a non-conjugate $\SL_n$--trace equivalent pair if and only if $F_s$ has a non-conjugate $\SL_n$--trace equivalent pair. In particular, we need only consider the existence of trace equivalent pairs in $F_2$. In fact, for any finitely generated group $\Gamma$ with a malnormal free subgroup, we see that $\Conj_{F_2}(n) \preceq \Conj_\Gamma(n)$. Moreover, $\Conj_{F_r}(n) \approx \Conj_{F_s}(n)$ for any integers $r,s \geq 2$. We also note that Lemma \ref{L:Mal} implies that if $F_2$ does not have (A), then no non-elementary hyperbolic group can have (A). Indeed, no finitely generated group with a malnormal free subgroup can have (A).

\subsection{Horowitz's construction}

Let $F_2=\innp{a,b}$.  Horowitz's words are defined recursively by $w_0=a$ and 
\[ w_m(\epsilon_1,\dots,\epsilon_m):=w^{-\epsilon_m}_{m-1}b^{2m}w^{\epsilon_m}_{m-1}b^{2m-1}w^{-\epsilon_m}_{m-1}b^{2m}w^{\epsilon_m}_{m-1} \] 
for $\epsilon_i=\pm 1$. Horowitz shows that for $(\epsilon_1,\dots,\epsilon_m)\not=(\epsilon_1^*,\dots,\epsilon_m^*)$, the corresponding words will not be cyclically equivalent for any $m>0$ and they are all $\SL_2$--trace equivalent. Hence, there are arbitrarily large collections of  $\SL_2$--trace equivalent non-conjugate words. For $w_1(1)=a^{-1}b^2aba^{-1}b^2a$ and $w_1(-1)=ab^2a^{-1}bab^2a^{-1}$, one can find a representation $\rho=(A,B)\in \SL(3,\C)^2$ where $\Tr(\rho(w_1(1)))-\Tr(\rho(w_1(-1)))\ne 0$. In particular, this pair is not $\SL_3$--trace equivalent. Below, we further elaborate on why it is unlikely that the above $\SL_2$--trace pairs are also $\SL_3$--trace pairs. First, we review in more detail why these pairs are $\SL_2$--trace equivalent. The first step in showing $\SL_2$--trace equivalence is a proof that 
\[ \Tr(w_m(\epsilon_1,\dots,\epsilon_{j-1},+1,\epsilon_{j+1},\dots,\epsilon_m))=\Tr(w_m(\epsilon_1,\dots,\epsilon_{j-1},-1,\epsilon_{j+1},\dots,\epsilon_m)) \] 
for $1\leq j\leq m$ for $\SL(2,\C)$.  By the recursive definition of $w_m$, Horowitz shows that 
\[ w_m(\epsilon_1,\dots,\epsilon_{j-1},+1,\epsilon_{j+1},\dots,\epsilon_m)=W(u^{-1}bu,b) \]
whereas 
\[ w_m(\epsilon_1,\dots,\epsilon_{j-1},-1,\epsilon_{j+1},\dots,\epsilon_m)=W(ubu^{-1},b) \] 
where $u=w_{j-1}(\epsilon_1,\dots,\epsilon_{j-1})$ and $W$ is a word in two letters. What works for $\SL(2,\C)$ is that there exists a polynomial $P_W$ in three variable so that $\Tr(W(u,v))=P_W(\Tr(u),\Tr(v),\Tr(uv))$. In the case above, these three traces are identical when evaluated at $(ubu^{-1},b)$ and $(u^{-1}bu,b)$ respectively since the trace is invariant under cyclic permutations, and hence their polynomials are equal too.  One can argue inductively to establish the general result.

However, this first step fails for $\SL(3,\C)$. The comparable statement is that there exists a polynomial $P_W$ in nine variables (see \cite{Lawton}) such that 
\[ \Tr(W(u,v))= P_W(\Tr(u),\Tr(u^{-1}),\Tr(v),\Tr(v^{-1}),\Tr(uv),\Tr(u^{-1}v^{-1}),\Tr(uv^{-1}),\Tr(u^{-1}v),\Tr(uvu^{-1}v^{-1})). \]  
Upon checking, one finds that the first 6 variables are equal. However, the seventh variables become $\Tr(ubu^{-1}b^{-1})$ and $\Tr(u^{-1}bub^{-1})=\Tr(bub^{-1}u^{-1})=\Tr((ubu^{-1}b^{-1})^{-1})$. These traces of words are generically not equal (see \cite{Lawton}); in fact they are equal if and only if the $\SL(3,\C)$ representations are transpose fixed.  Likewise the eighth variables will differ as well.  In fact, we would have an expression of the form $P_W(a_1,\dots,a_6,a_7,a_8,\Tr(w))=P_W(a_1,\dots,a_6,a_8,a_7,\Tr(w^{-1}))$ since the 7th and 8th variables are in fact permuted (switching the roles of $u$ and $v$) and the first 6 are identical (just by cyclic permutation), and the 9th is cyclically equivalent to the trace of its inverse.  The 9th word is $ubu^{-1}bub^{-1}u^{-1}b^{-1}$.  Note also that there is a polynomial $P$ in the 8 algebraically independent variables so that $\Tr(w^{-1})=P-\Tr(w)$. If we had equality we would have a non-trivial relation (symmetric in two variables), which is unlikely for a fixed $w$.

\subsection{Candidate words}

By Lemma 6.8 in \cite{Horowitz} any $\SL_2$--trace equivalent pair in $F_2$ must have the same number of each generator represented in the word, up to plus or minus exponents.  Thus, the same result holds for words that are $\SL_n$--trace equivalent for any $n$. It is easy to see that pairs of the form $(w,w^{-1})$ are $\SL_2$--trace equivalent.  However, by \cite{Bo83} the word map is dominant for non-trivial words, and so $(w,w^{-1})$ are never $\SL_n$--trace equivalent for $n\geq 3$ since $\Tr(A)\not=\Tr(A^{-1})$ for a generic (in the Baire sense) $A\in\SL(3,\C)$. Along the same lines, we have the following lemma.

\begin{lemma}\label{reverselemma}
Let $r(w)$ be the reverse of the word $w$, and assume $r(w)$ is not conjugate to $w$.  Then $r(w)$ and $w$ are always $\SL_n$--trace equivalent if and only if $n=2$. 
\end{lemma}

\begin{proof}
Since $\Tr(w)=\Tr(w^{-1})$ for $n=2$, we obtain $\Tr(w(a,b))=\Tr(w(a,b)^{-1})=\Tr(r(w(a^{-1},b^{-1})))$. Therefore, $\Tr(r(w(a,b)))=\Tr(w(a^{-1},b^{-1}))$. By the Fricke--Vogt Theorem (see for instance \cite{G9}), $\Hom(F_2,\SL(2,\C))/\!\!/ \SL(2,\C)\cong \C^3$ parametrized by $(\Tr(a),\Tr(b),\Tr(ab))$.  Thus, there exists a unique polynomial $P\in \C[x,y,z]$ such that $\Tr(w(a,b))=P(\Tr(a),\Tr(b),\Tr(ab))$. We conclude 
\[ \Tr(r(w(a,b)))=P(\Tr(a^{-1}),\Tr(b^{-1}),\Tr(a^{-1}b^{-1}))=P(\Tr(a),\Tr(b),\Tr(ab))=\Tr(w(a,b)). \]
Conversely, $\Hom(F_2,\SL(3,\C))/\!\!/ \SL(3,\C)$ is a branched double cover of $\C^8$ (see \cite{Lawton}).  The branch locus is exactly determined by $\Tr(aba^{-1}b^{-1})=\Tr(b^{-1}a^{-1}ba)$; showing that for $r=2$ the pairs $(w,r(w))$ are not generally $\SL_n$--trace equivalent for $n\geq 3$.
\end{proof}

We expect that non-conjugate reverse pairs are never $\SL_3$--trace equivalent. A more provocative conjecture is the following; in the statement, positive words have only non-negative powers of the generators:

\begin{conj}
Let $n\geq 2$.  There exists $\SL_n$--trace equivalent pairs $(u,v)$ if and only if there exists positive pairs $(u',v')$ that are $\SL_n$--trace equivalent.
\end{conj}

Before giving a heuristic proof for the above conjecture, we mention two related conjectures. Ginzburg--Rudnick \cite[Conj~1.1]{GR} have a conjectural condition to ensure a word does not have an $\SL_2$--trace companion (aside from its inverse); in their terminology, such a word has stable multiplicity one. Anderson \cite[Conj~4.1]{Anderson} gave conjectural picture for all $\SL_2$--trace companions.

We now give a heuristic for the validity of the conjecture. As the reverse implication is obvious, we discuss only the direct implication. For $n=2$, Lemma \ref{reverselemma} establishes the statement.  For $n>2$ we describe an algorithm (that depends on $n$) that takes a non-conjugate $\SL_n$--trace equivalent pair and produces a pair, that we expect that is positive, $\SL_n$--trace equivalent, and not conjugate.  We have implemented the algorithm for $n=2$ and it does produce a positive pair $(u',v')$ that is $\SL_2$--trace equivalent but $u'$ is conjugate to $v'$; we expect this to be a problem only with $n=2$.

In what follows, let $\rho(a)=\ab$ be a $n$ by $n$ matrix.  Recall the Cayley--Hamilton formula gives 
\[ \mathbf{0}=\sum_{k=0}^{n} (-1)^{n-k}C^n_k(\ab)\ab^k, \] 
where the coefficients $C^n_k(\ab)$ arise from the characteristic equation $\det(t\id-\ab)=\sum_{k=0}^{n} (-1)^{n-k}C^n_k(\ab)t^k$.
We know that $C^n_n(\ab)=1$, $C^n_{n-1}(\ab)=\Tr(\ab)$ and $C^n_0(\ab)=\det(\ab)$.  By Newton's trace formulas each $C^n_k(\ab)$ is a polynomial in the traces of non-negative powers of the matrix $\ab$.  Since $\det(\ab)=1$, we can multiply the Cayley-Hamilton formula by a word $\mathbf{U}\ab^{-1}:=\rho(ua^{-1})$ on the left and another word $\mathbf{V}:=\rho(v)$ on the right.  This results in 
\begin{equation*}\label{noncomtradeup}\mathbf{U}\ab^{n-1}\mathbf{V}+\sum_{k=1}^{n-1} (-1)^{n-k}C^n_k(\ab)\mathbf{U}\ab^{k-1}\mathbf{V}=(-1)^{n+1}\mathbf{U}\ab^{-1}\mathbf{V}.\end{equation*}  
Thus, by taking the trace of both sides, we have:  
\begin{equation*}\label{tradeup}\Tr(\mathbf{U}\ab^{-1}\mathbf{V})=(-1)^{n+1}\Tr(\mathbf{U}\ab^{n-1}\mathbf{V})+\sum_{k=1}^{n-1} (-1)^{k-1}C^n_k(\ab)\Tr(\mathbf{U}\ab^{k-1}\mathbf{V}).\end{equation*}
That shows that given any word $w$ with negative exponents, one can iteratively apply the preceding formula in the coordinate ring $\C[\Hom(F_2,\SL(n,\C))/\!\!/\SL(n,\C)]$, which is generated by traces of words by results of Procesi \cite{P1}, to obtain an expression for $\Tr(w)$ as a polynomial in traces of positive words.

Now, suppose $(u,v)$ is $\SL_n$--trace equivalent but are not conjugate.  After cyclically reducing $u$ and $v$, given results of Horowitz (\cite{Horowitz}), we can assume that $u$ and $v$ have the same word length and the same (signed) multiplicity of each letter. Applying the preceding algorithm to $\Tr(u), \Tr(v)$ results in polynomial expressions $P_u, P_v$ in terms of traces of only positive words.  By inspection of the replacement formula defining the algorithm, one sees that there will be a monic trace term with a longest word.  That is $P_u=\Tr(u')+L$, and likewise $P_v=\Tr(v')+L'$ where both $L,L'$ contain terms of products of traces of shorter positive words. We expect that $\Tr(u')=\Tr(v')$ since $\Tr(u)=\Tr(v)$ to begin with.  Also, given that $n\geq 3$, we expect that $u'$ is not conjugate to $v'$ given that $u$ is not conjugate to $v$.

It is not presently clear to us how to complete the above argument, that is, to prove that the last two lines are valid.  We thank Greg Kuperberg for conversations about the validity of the above sketch.

We now indicate our interest in this conjecture. For the free group $F_2 = F_2(a,b)$, the smallest positive exponent $\SL_2$--trace equivalent pair is $\set{babbaa, abaabb}$.
To find examples of $\SL(3,\C)$ words, if the conjecture is true, we need only check words with the same number of letters in each word having only positive exponents.  Moreover, since by restricting, the trace equivalence must also hold for $\SL(2,\C)$, we need only check words of the above type that work for $\SL(2,\C)$.  We expect that non-conjugate reverse pairs will never be $\SL_3$--trace equivalent, and so we further wish to only consider positive non-conjugate pairs that are not reverse but are $\SL_2$--trace equivalent; the first examples occurs at length 12 with one explicit pair being $\set{aababbaabbab, aababbabaabb}$. We end this section with two questions about such words.
\begin{enumerate}
 \item[(1)] What is a classification of these words, or generating families?
 \item[(2)] What is the growth rate as a function the length of these words? 
 \end{enumerate}
As we expect $\SL_n$--trace equivalent words exist, our guess is that the above words are rather plentiful. However, by computer search, there are no $\SL_3$--trace equivalent pairs of length up to $20$.

\section{Efficient solutions to the conjugacy problem}\label{Algor}

In this section, we provide two different approaches to solving the conjugacy problem in free groups using finite quotients, neither of which are originally due to us. 

\subsection{Lower central and derived series}

Recall, the lower central and derived series are defined inductively by $\Gamma_0 = \Gamma$, $\Gamma_j = [\Gamma,\Gamma_{j-1}]$, and $\Gamma^j = [\Gamma^{j-1},\Gamma^{j-1}]$. We set $N_j(\Gamma) = \Gamma/\Gamma_j$ and $S_j(\Gamma)=\Gamma/\Gamma^j$. By \cite[p. 27, Prop.~4.9]{LS}, we know that $\gamma,\eta \in F_r$ are conjugate in $F_r$ if and only if they have conjugate image in $S_j(F_r)$ (or $N_j(F_r)$) for all $j$. Since the groups $S_j(F_r)$ and $N_j(F_r)$ are conjugacy separable for all $j$ (see \cite{Blackburn}, \cite{Formanek}, and \cite{Rem}) and we see that $F_r$ conjugacy separable. In order to implement these methods effectively, we must first estimate $j_{\gamma,\eta}$ as a function of the word length of $\gamma,\eta$ where $j_{\gamma,\eta}$ is the smallest $j \in \N$ such that $\gamma,\eta$ have non-conjugate images in $N_j(F_r)$ (or $S_j(F_r)$). Second, we must effectivly solve the conjugacy problem in torsion free nilpotent or polycyclic groups. Malestein--Putman \cite{MP} addresses the first problem. Pengitore \cite{Peng} addresses the second problem. As our current goal is deciding whether or not the function $\Conj_{F_r}(n)$ has a polynomial bound, we note that it is already known that the above method cannot work. Specifically, neither the lower central or derived series provides a polynomial upper bound for the function $\WDep_{F_r}(n)$; see \cite{BouRabee10} and \cite{BouRabee11}. 

\subsection{Proof of Theorem \ref{FreeHaveB}}

In this subsection, we prove Theorem \ref{FreeHaveB}. The construction of the representation needed to verify Theorem \ref{FreeHaveB} in the case of free groups follows Wehrfritz \cite{Wehrfritz}. The surface group case is similar.

We now produce the representation for the case of free groups. To begin, given a conjugacy class $[\gamma]$ in $F_r$, we first pass to a finite index subgroup $\Gamma$ where $\Gamma = \innp{\gamma} \ast \Delta$. That such can be done follows from work of Hall \cite{Hall}. Since $\gamma$ is part of a free basis, it follows that there exists a representation $\rho_0\colon \Gamma \to \SL(2,\R)$ where $\gamma$ has a unique, non-zero trace up to conjugation and inverses. As $\SL(2,\R) < \SL(3,\C)$ by the standard inclusion into the upper two by two block, we see that there exists a representation $\rho_1\colon \Gamma \to \SL(3,\C)$ such that $\rho_1(\gamma)$ has a unique, non-zero trace up to conjugation and inverses. Since $\Tr(\rho(\gamma^{-1})) \ne \Tr(\rho(\gamma))$ for a generic $\SL(3,\C)$ representation $\rho$ (in the Baire Category sense), we can further assume that $\rho(\gamma)$ has a unique, non-zero trace up to conjugation. For any $\eta \in \Gamma$ that is not conjugate to $\gamma$ in $\Gamma$, we know that $\Tr(\rho(\gamma)) - \Tr(\rho(\eta))$ is a non-constant function of $\rho$. Consequently, by the Baire Category Theorem, we can assume that $k_1\Tr(\rho_1(\gamma)) \ne k_2\Tr(\rho_1(\eta))$ for any pair of integers $1 \leq k_1,k_2 \leq m = [F_r:\Gamma]$. For any such $\rho_1 \in \Hom(\Gamma,\SL(3,\C))$, the induced representation $\rho = \mathrm{Ind}_\Gamma^{F_r}(\rho_1)$ is the needed representation to verify Theorem \ref{FreeHaveB} in the free group case. 

\begin{proof}[Proof of Theorem \ref{FreeHaveB}: Free Case]
If $\eta \in F_r$ is not conjugate in $F_r$ into $\Gamma$ , then $\Tr(\rho(\eta)) = 0$ by the Frobenius formula for traces of induced representations. If $\eta \in F_r$ is conjugate in $F_r$ to some $\eta' \in \Gamma$, then $\Tr(\rho(\eta)) = k_\eta \Tr(\rho_1(\eta'))$ for some integer $1 \leq k_\eta \leq m$. As $\Tr(\rho(\gamma)) = k_\gamma \Tr(\rho_1(\gamma))$ for some $1 \leq k_\gamma \leq m$, it follows from our selection of $\rho_1$ that $\rho(\gamma)$ has a unique trace up to conjugation, as needed for Theorem  \ref{FreeHaveB}.
\end{proof}

We now produce the representation for the case of surface groups. To begin, given a conjugacy class $[\gamma]$ in $\pi_1(\Sigma_g)$, we first pass to a cover where a lift of the curve associated to $[\gamma]$ is simple. That such can be done follows from work of Scott \cite{Scott}. We fix a finite index subgroup of $\pi_1(\Sigma_g)$ associated to this finite cover which we denote by $\Gamma$. Since the curve associated to $[\gamma]$ has a simple lift, it follows that there exists a representation $\rho_0\colon \Gamma \to \SL(2,\C)$ where $\gamma$ has a unique, non-zero trace up to conjugation and inverses. The remainder of the construction of $\rho$ is identical to the free case of the proof of Theorem \ref{FreeHaveB}. Note that to ensure $\Tr(\rho(\gamma)) \ne \Tr(\rho(\gamma^{-1}))$ for a generic $\rho \in \Hom(\Gamma,\SL(3,\C)$ (in the Baire sense), we can use \cite{BLa} in place of \cite{Bo83}. 

In either the free or surface case, we can use the methods from  \cite{BM12} to establish Corollary \ref{FreeRelComplexity}. In particular, the degree of the polynomial in Corollary \ref{FreeRelComplexity} depends only on $m$ and the coefficient ring of the representation, both of which are constant for a fixed $\gamma$. By Patel \cite{Patel} and Gupta--Kapovich \cite{GK}, we have $m \preceq \norm{\gamma}$, and so when $\norm{\gamma},\norm{\eta} \leq n$, we see that $\Con_\Gamma(\gamma,\eta) \leq Cn^{Cn^2}$.

\section{Proof of Theorem \ref{RepresentationTopology}}\label{RepTop}

Given a fully residually free group $\Gamma$ with a finite index, normal subgroup $\Delta$ and a prime $p \in \Z$, we will construct a faithful homomorphism $\rho\colon \Gamma \to \SL(n_\Delta,R_\omega)$ such that $\Delta = \ker(r_{\mathfrak{m}_\omega} \circ \rho)$ where $n_\Delta = 2[\Gamma:\Delta]$, $R_\omega$ is a local domain, $\mathfrak{m}_\omega <R_\omega$ is the unique maximal ideal with residue field $R_\omega/\mathfrak{m}_\omega = \mathbf{F}_p$, the field of $p$ elements, and $r_{\mathfrak{m}_\omega}\colon \SL(n_\Delta,R_\omega) \to \SL(n_\Delta,\mathbf{F}_p)$ is the reduction modulo $\mathfrak{m}_\omega$ homomorphism. We enumerate the non-trivial elements of $\Delta$ via $\set{\delta_1,\delta_2,\dots}$. Since subgroups of fully residually free groups are fully residually free, for each $t \in\N$, there exists a homomorphism $\psi_t\colon \Delta \to F_{r_t}$ such that $\psi_t$ is injective when restricted to the finite subset $\set{\delta_1,\dots,\delta_t}$. Recall that the ring of $p$--adic integers $\Z_p$ is a local integral domain with a unique maximal ideal $\mathfrak{m_p}$. Via the ping pong lemma, the homomorphism $\psi_{(p)}\colon F_2 \to \SL(2,\Z) < \SL(2,\Z_p)$ induced by sending a free basis $a,b$ of $F_2$ to the matrices
\[ a,b \longmapsto \begin{pmatrix} 1 & p \\ 0 & 1 \end{pmatrix}, \begin{pmatrix} 1 & 0 \\ p & 1 \end{pmatrix} \]
is an isomorphism.  By the Nielsen--Schreier theorem, we have a faithful homomorphism $F_{r_t} \to F_2$ for each $r_t \in \N$ and fix one such homomorphism for each $r_t \in \N$. Respectfully, we define $\rho_{p,t}:= \psi_{(p)}\circ \psi_t$ and note $\rho_{p,t}(\Delta) < \ker r_{\mathfrak{m}_p}$.

We restrict $\psi_{(p)}$ to the image of $F_{r_t} < F_2$ and for notational simplicity denote the resulting homomorphism by $\psi$.  Taking a non-principal ultrafilter $\omega$ on $\N$, the ultraproduct $R_\omega = \prod_\omega \Z_p$ is a local integral domain with unique maximal ideal $\mathfrak{m}_\omega = \prod_\omega \mathfrak{m}_p$ (see \cite[Ch. 1]{Scho} for instance). The associated residue field $R_\omega/\mathfrak{m}_\omega$ is given by $\prod_\omega \Z_p/\mathfrak{m}_p$. Since the latter is an ultraproduct of $\mathbf{F}_p$, it follows that $R_\omega/\mathfrak{m}_\omega$ is isomorphic to $\mathbf{F}_p$ (see \cite[p. 184]{Hal95} for instance). The ultraproduct $\rho_\omega$ of the representations $\rho_t = \psi \circ \psi_t$ yields a representation $\rho_\omega\colon \Delta \to \SL(2,R_\omega)$. By selection of $\psi_t$ and $\psi$, $\rho_\omega$ is faithful with $\rho_\omega(\Delta) < \ker r_{\mathfrak{m}_\omega}$. Setting $\rho = \mathrm{Ind}_\Delta^\Gamma(\rho_\infty)$, we obtain a faithful representation $\rho\colon \Gamma \to \SL(2d,R_\omega)$ where $d=[\Gamma:\Delta]$. By construction of $\rho_\omega$, the definition of $\mathrm{Ind}$, and the normality of $\Delta$ in $\Gamma$, we see that $\Delta = \ker(r_{\mathfrak{m}_\omega} \circ \rho)$.\qed 

\begin{rem}\label{R:Ring}
The ring $R_\omega$ embeds into $\prod_\omega \Q_p$ which is a field of characteristic zero. Since fully residually free groups are finitely presentable (\cite[4.4]{Sela}), the ring $R$ generated over $\Z$ by the coefficients of the matrix entries of $\rho(\Gamma)$ is finitely generated. Setting $\mathfrak{m} = R \cap \mathfrak{m}_\omega$, we obtain a maximal ideal in $R$ with residue field $R/\mathfrak{m} = \mathbf{F}_p$ such that $\rho(\Gamma) < \SL(2[\Gamma:\Delta],R)$ and $\Delta = \ker(r_{\mathfrak{m}} \circ \rho)$. Moreover, we have an embedding of $R$ into $\C$; the field of fractions of $R$ embeds into $\C$ via the axiom of choice. 
\end{rem}

\begin{rem}\label{R:Free}
If $\Gamma$ is a free group, we can take $\rho = \mathrm{Ind}_\Delta^\Gamma(\rho_0)$ where $\rho_0$ is the representation given by $\Delta \to F_2 \to \SL(2,\Z)$. The first homomorphism $\Delta\to F_2$ is given by the Nielsen--Schreier theorem and the second homomorphism $F_2 \to \SL(2,\Z)$ is given by $\psi_{(p)}$. In total, we obtain a faithful representation $\rho\colon \Gamma \to \SL(2[\Gamma:\Delta],\Z)$ such that $\Delta = \ker(r_p \circ \rho)$.
\end{rem}

Theorem \ref{RepresentationTopology} can also be proven by using work of Barlev--Gelander \cite{BG}, which followed the work of Breuillard--Gelander--Souto--Storm \cite{BGSS}. Barlev--Gelander \cite[Thm 1.2]{BG} proved that if $G$ is a compact topological group with a non-abelian free subgroup, then $G$ contains an isomorphic copy of every non-abelian limit group. Since $\Z_p$ is a compact topological ring, $\SL(2,\Z_p)$ is a compact topological group. Moreover, the finite index subgroup $\ker r_{\mathfrak{m}_p} < \SL(2,\Z_p)$ is a compact topological group with $F_2 < \ker r_{\mathfrak{m}_p}$ from above. Hence by \cite[Thm 1.2]{BG}, $\ker r_{\mathfrak{m}_p}$ contains an isomorphic copy of every non-abelian limit group. Given a non-abelian limit group $\Gamma$ with a finite index, normal subgroup $\Delta$, we can apply this observation to obtain a faithful representation $\rho_0\colon \Delta \to \ker r_{\mathfrak{m}_p} < \SL(2,\Z_p)$. It follows then that $\rho = \mathrm{Ind}_\Delta^\Gamma(\rho_0)$ is a faithful representation into $\SL(2[\Gamma:\Delta],\Z_p)$ with $\Delta = \ker(r_{\mathfrak{m}_p} \circ \rho)$.


\small{
 Lawton: George Mason University, Fairfax, VA 22030. \verb"slawton3@gmu.edu" 

 Louder: University College London, London, UK, WC1E 6BT. \verb"l.louder@ucl.ac.uk" 

 McReynolds: Purdue University, West Lafayette IN 47907. \verb"dmcreyno@purdue.edu"}



\begin{thebibliography}{9999}

\small{

\bibitem{Anderson}
J.~W.~Anderson, \emph{Variations on a theme of Horowitz}, Kleinian groups and hyperbolic 3-manifolds (Warwick, 2001), 307--341, London Math. Soc. Lecture Note Ser., 299, Cambridge Univ. Press, Cambridge, 2003.

\bibitem{BL}
H.~Bass, A.~Lubotzky, \emph{Automorphism of schemes and of subgroups of finite type}, Israel J. of Math. \textbf{44} (1983), 1--22.

\bibitem{BG}
J.~Barlev, T.~Gelander, \emph{Compactifications and algebraic completions of limit groups}, Journal d'Analyse Math{\'e}matique \textbf{112} (2010), 261--287.

\bibitem{BLR}
I.~Biswas, S.~ Lawton, D.~Ramras, \emph{Fundamental groups of character varieties: surfaces and tori}, Math. Z. \textbf{281} (2015), 415--425.

\bibitem{Blackburn}
N.~Blackburn, \emph{Conjugacy in nilpotent groups}, Proc. Amer. Math. Soc. \textbf{16} (1965), 143--148.

\bibitem{Bo83}
A.~Borel, \emph{On free subgroups of semisimple groups}, Enseign. Math. (2), \textbf{29} (1983), 151--164.

\bibitem{Borel}
A.~Borel, \emph{Linear algebraic groups}, Springer-Verlag, 1991.

\bibitem{BouRabee10}
K.~Bou-Rabee, \emph{Quantifying residual finiteness}, J. Algebra \textbf{323} (2010), 729--737.

\bibitem{BouRabee11}
K.~Bou-Rabee, \emph{Approximating a group by its solvable quotients}, New York J. Math. \textbf{17} (2011), 699--712.

\bibitem{BHP}
K.~Bou-Rabee, M.~F.~Hagen, P.~Patel, \emph{Residual finiteness growths of virtually special groups}, Math. Z. \textbf{279} (2015), 297--310.

\bibitem{BK}
K.~Bou-Rabee, T.~Kaletha, \emph{Quantifying residual finiteness of arithmetic groups}, Compos. Math. \textbf{148} (2012), 907--920.

\bibitem{BLa}
K.~Bou-Rabee, M.~Larsen, \emph{Linear groups with Borel's property}, \href{http://arxiv.org/abs/1312.7294}{ArXiv}, to appear in J. Eur. Math. Soc.

\bibitem{BM10}
K.~Bou-Rabee, D.~B.~McReynolds, \emph{Bertrand's postulate and subgroup growth}, J. of Algebra \textbf{324} (2010), 793--819.

\bibitem{BM11}
K.~Bou-Rabee, D.~B.~McReynolds, \emph{Asymptotic growth and least common multiples in groups}, Bull. Lond. Math. Soc. \textbf{43} (2011), 1059--1068. 

\bibitem{BM12}
K.~Bou-Rabee, D.~B.~McReynolds, \emph{Extremal behavior of divisibility functions}, Geom. Dedicata. \textbf{175} (2015), 407--415.
\bibitem{BM14}
K.~Bou-Rabee, D.~B.~McReynolds, \emph{Characterizing linear groups in terms of growth properties}, \href{http://arxiv.org/pdf/1403.0983.pdf}{ArXiv}, to appear in Michigan Math. J.

\bibitem{BS13}
K.~Bou-Rabee, B.~Seward, \emph{Arbitrarily large residual finiteness growth}, J. Reine Angew. Math. \textbf{710} (2016), 199--204.

\bibitem{BGSS}
E.~Breuillard, T.~Gelander, J.~Souto, P.~Storm, \emph{Dense embeddings of surface groups}, Geom. Topol. \textbf{10} (2006), 1373--1389. 

\bibitem{Buskin}
N.~V.~Buskin, \emph{Efficient separability in free groups}, Sibirsk. Mat. Zh. \textbf{50} (2009), 765--771.

\bibitem{CBW}
O.~Cotton-Barratt, H.~Wilton, \emph{Conjugacy separability of 1-acylindrical graphs of free groups}
Math. Z. \textbf{272} (2012), 1103--1114.

\bibitem{Dehn}
M.~Dehn, \emph{\"Uber unendliche diskontinuierliche Gruppen}, Math. Ann. \textbf{71} (1911), 116--144.

\bibitem{FL}
C.~Florentino, S.~Lawton, \emph{Topology of character varieties of Abelian groups}, Topology Appl. \textbf{173} (2014), 32--58. 

\bibitem{Formanek}
E.~Formanek, \emph{Conjugate separability in polycyclic groups}, J. Algebra \textbf{42} (1976), 1--10.

\bibitem{Gas}
J.~Gaster, \emph{Lifting curves simply}, \href{http://arxiv.org/abs/1501.00295}{ArXiv}, to appear in IMRN.

\bibitem{GR}
D.~Ginzburg, Z.~Rudnick, \emph{Stable multiplicities in the length spectrum of Riemann surfaces} Israel J. of Math. \textbf{104} (1998), 129--144.

\bibitem{Goldman}
W.~M.~Goldman, \emph{Topological components of spaces of representations}, Invent. Math. \textbf{93} (1988), 557--607. 

\bibitem{G9}
W.~M. Goldman, \emph{Trace coordinates on Fricke spaces of some simple hyperbolic
  surfaces}, EMS Publishing House, Z\"urich, 2008. Handbook of Teichm\"uller theory II ( A. Papadopoulos, editor).

\bibitem{GK}  
N.~Gupta, I.~Kapovich, \emph{The primitivity index function for a free group, and untangling closed curves on hyperbolic surfaces}, \href{http://arxiv.org/abs/1411.5523}{ArXiv} (2014).

\bibitem{Hal95}
J.~I.~Hall, \emph{Locally finite simple groups of finitary linear transformations}, Nato Science Series C \textbf{471} (1995), 147--188.

\bibitem{Hall}
M.~Hall, \emph{Coset representations of free groups}, Trans. Amer. Math. Soc. \textbf{67} (1949), 421--432.

\bibitem{HWZ}
E.~Hamilton, H.~Wilton, P.~Zalesskii, \emph{Separability of double cosets and conjugacy classes in 3-manifold groups},
J. Lond. Math. Soc. \textbf{87} (2013), 269--288.

\bibitem{Horowitz}
R.~Horowitz, \emph{Characters of free groups represented in the two-dimensional special linear group}, Comm. Pure Appl. Math. \textbf{25} (1972), 635--649.

\bibitem{Kap}
I.~Kapovich, \emph{A non-quasiconvexity embedding theorem for hyperbolic groups} Math. Proc. Cambridge Philos. Soc. \textbf{127} (1999), 461--486.

\bibitem{KLSS}
I.~Kapovich, G.~Levitt, P.~Schupp, V.~Shpilrain, \emph{Translation equivalence in free groups}, Trans. Amer. Math. Soc. \textbf{359} (2007), 1527--1546.

\bibitem{KM}
M.~Kassabov, F.~Matucci, \emph{Bounding the residual finiteness of free groups}, Proc. Amer. Math. Soc. \textbf{139} (2011), 2281--2286. 

\bibitem{KMS}
O.~Kharlampovich, A.~Myasnikov, M.~Sapir, \emph{Algorithmically complex residually finite groups}, \href{http://front.math.ucdavis.edu/1204.6506}{ArXiv} (2012).

\bibitem{KT}
G.~Kozma, A.~Thom, \emph{Divisibility and laws in finite simple groups}, Math. Ann. \textbf{364} (2016), 79--95.

\bibitem{Kuperberg}
G.~Kuperberg, \emph{Knottedness is in NP, modulo GRH}, Adv. Math. \textbf{256} (2014), 493--506

\bibitem{Lawton}
S.~Lawton, \emph{Generators, relations and symmetries in pairs of {$3\times 3$} unimodular matrices}, J. Algebra \textbf{313} (2007), 782--801.

\bibitem{Lee}
D.~Lee, \emph{Translation equivalent elements in free groups}, J. Group Theory \textbf{9} (2006), 809--814. 

\bibitem{Lee2}
D.~Lee, \emph{An algorithm that decides translation equivalence in a free group of rank two}, J. Group Theory \textbf{10} (2007), 561--569.

\bibitem{LV}
D.~Lee, E.~Ventura, \emph{Volume equivalence of subgroups of free groups}, J. Algebra \textbf{324} (2010), 195--217. 

\bibitem{Leininger}
C.~J.~Leininger, \emph{Equivalent curves in surfaces}, Geom. Dedicata \textbf{102} (2003), 151--177.

\bibitem{LR}
D.~D.~Long, A.~W.~Reid, \emph{Integral points on character varieties}, Math. Ann. \textbf{325} (2003), 299--321.

\bibitem{LMP}
L.~Louder, D.~B.~McReynolds, P.~Patel, \emph{Zariski Closures and Subgroup Separability}, \href{http://arxiv.org/abs/1510.04144}{ArXiV} (2015).

\bibitem{LS}
R.~C.~Lyndon, P.~E.~Schupp, \emph{Combinatorial group theory}, Springer--Verlag, 1977.

\bibitem{MalcevRF}
A.~I.~Mal'cev, \emph{On the faithful representation of infinite groups by matrices}, Mat. Sb. \textbf{8} (1940), 405--422.

\bibitem{MalcevCS}
A.~I.~Mal'cev, \emph{On homomorphisms onto finite groups}, Uchen. Zap. Ivanovskogo Gos. Ped. Inst. \textbf{18} (1958), 49--60.

\bibitem{MP}
J.~Malestein, A.~Putman, \emph{On the self-intersections of curves deep in the lower central series of a surface group}, Geom. Dedicata \textbf{149} (2010), 73--84.

\bibitem{Paris}
L.~Paris, \emph{Residual $p$ properties of mapping class groups and surface groups}, Trans. Amer. Math. Soc. \textbf{361} (2008), 2487--2507

\bibitem{Patel}
P.~Patel, \emph{On a theorem of Peter Scott}, Proc. Amer. Math. Soc. \textbf{142} (2014), 2891--2906. 

\bibitem{Patel2}
P.~Patel, \emph{On the residual finiteness growths of particular hyperbolic manifold groups}, \href{http://arxiv.org/abs/1412.6835}{ArXiV}, to appear in Geom. Dedicata.

\bibitem{Peng}
M.~Pengitore, \emph{Effective conjugacy separability of finitely generated nilpotent groups}, \href{http://arxiv.org/abs/1502.05445}{ArXiV} (2015).

\bibitem{PLR}
V.~Platonov, A.~Rapinchuk, \emph{Algebraic groups and number theory}, Academic Press, 1994.

\bibitem{P1}
C.~Procesi, \emph{The invariant theory of {$n\times n$} matrices}, Adv. Math. \textbf{19} (1976), 306--381.

\bibitem{Rag}
M.~S.~Raghunathan, \emph{Discrete subgroups of Lie groups}, Springer--Verlag, 1972.

\bibitem{RBC}
A.~Rapinchuk, V.~ Benyash-Krivetz, V.~Chernousov, \emph{Representation varieties of the fundamental groups of compact orientable surfaces}, Israel J. Math. \textbf{93} (1996), 29--71. 

\bibitem{Rem}
V.~N.~Remeslennikov, \emph{Conjugacy in polycyclic groups}, Akademiya Nauk SSSR. Sibirskoe Otdelenie. Institut Matematiki. Algebra i Logika, \textbf{8} (1969), 712--725.

\bibitem{Rivin}
I.~Rivin, \emph{Geodesics with one self-intersection, and other stories}, Adv. Math. \textbf{231} (2012), 2391--2412.

\bibitem{Scho}
H.~Schoutens, \emph{The Use of Ultraproducts in Commutative Algebra}, Springer--Verlag, 2010.

\bibitem{Schreier}
O.~Schreier, \emph{Die Untergruppen der freien Gruppe}, Abhandlungen aus dem Mathematischen Seminar der Universit\"{a}t Hamburg \textbf{5} (1927), 161--183.

\bibitem{Schwarz}
G.~W.~Schwarz, \emph{The topology of algebraic quotients}, Progr. Math. \textbf{80} (1989), 135--151.

\bibitem{Scott}
G.~P.~Scott, \emph{Subgroups of surface groups are almost geometric}, J. London Math. Soc. \textbf{17} (1978) 555--565;
J. London Math. Soc. \textbf{32} (1985) 217--220.

\bibitem{Sela}
Z.~Sela, \emph{Diophantine geometry over groups. I. Makanin-Razborov diagrams}, Publ. Math. Inst. Hautes \'{E}tudes Sci. \textbf{93} (2001), 31--105.

\bibitem{Stebe}
P.~F.~Stebe, \emph{Conjugacy separability of groups of integer matrices}, Proc. Amer. Math. Soc. \textbf{32} (1972), 1--7. 

\bibitem{Toi}
E.~Toinet, \emph{Conjugacy $p$--separability of right-angled Artin groups and applications}, Groups Geom. Dyn. \textbf{7} (2013), 751--790.

\bibitem{Thom}
A.~Thom, \emph{About the length of laws for finite groups}, \href{http://arxiv.org/abs/1508.07730}{ArXiv} (2015).

\bibitem{Wehrfritz}
B.~A.~F.~Wehrfritz, \emph{Conjugacy separating representations of free groups}, Proc. Amer. Math. Soc. \textbf{40} (1973), 52--56. 

\bibitem{Wilton}
H.~Wilton, \emph{Virtual retractions, conjugacy separability and omnipotence}, J. Algebra \textbf{323} (2010), 323--335.

}
\end{thebibliography}
\end{document}